\documentclass[12pt]{amsart}
\usepackage{amsmath,amssymb}
\usepackage[utf8,utf8x]{inputenc}
\usepackage[pagebackref=true,
pdftitle = {Free\ but\ not\ recursively\ free\ arrangements},
pdfkeywords = {arrangement\ of\ hyperplanes; free\ arrangement; Terao's\ conjecture;},
pdfsubject = {52C35; 20F55; 14N20; 13N15},
pdfauthor = {}]
{hyperref}
\usepackage{hyperref}
\usepackage{longtable}
\usepackage{graphicx}
\usepackage{pict2e}

\newtheorem{propo}{Proposition}[section]

\newtheorem{theor}[propo]{Theorem}
\newtheorem{lemma}[propo]{Lemma}

\theoremstyle{definition}
\newtheorem{defin}[propo]{Definition}

\theoremstyle{remark}
\newtheorem{remar}[propo]{Remark}

\numberwithin{equation}{section}

\newcommand{\CC }{\mathbb{C}}

\newcommand{\FF }{\mathbb{F}}
\newcommand{\QQ }{\mathbb{Q}}

\newcommand{\Ac }{\mathcal{A}}

\newcommand{\Bc }{\mathcal{B}}

\DeclareMathOperator{\GL}{GL}
\DeclareMathOperator{\PGL}{PGL}

\DeclareMathOperator{\mm}{m}

\newcommand{\IF }{\mathcal{IF}}
\newcommand{\RF }{\mathcal{RF}}
\DeclareMathOperator{\Der}{Der}
\DeclareMathOperator{\pdeg}{pdeg}

\title[Free but not recursively free arrangements]
{Free but not recursively\\ free arrangements}

\author{M.~Cuntz}
\address{Michael Cuntz,
Fachbereich Mathematik,
Universit\"at Kai\-sers\-lau\-tern,
Postfach 3049,
D-67653 Kaiserslautern, Germany}
\email{cuntz@mathematik.uni-kl.de}

\author{T.~Hoge}
\address{Torsten Hoge,
Ruhr-Universit\"at Bochum,
Fakult\"at f\"ur Mathematik,
Universit\"atsstrasse 150,
D-44780 Bochum, Germany}
\email{torsten.hoge@rub.de}

\begin{document}

\begin{abstract}
We construct counterexamples to the conjecture that every free arrangement is recursively free in characteristic zero.
The intersection lattice of our smallest example has a realization over a finite field which is recursively free,
thus recursive freeness is not a combinatorial property of the intersection lattice of an arrangement.
\end{abstract}

\maketitle

\section{Introduction}

Motivated by his famous Addition-Deletion-Theorem (see \cite{p-hT-80} or Thm.\ \ref{adddel} below), Terao introduced the notion of \emph{inductive freeness} of an arrangement $\Ac$ (Def.\ \ref{def:indfree}), a property which implies the freeness of the module of derivations $D(\Ac)$. Inductive freeness is a purely combinatorial property of the intersection lattice of an arrangement, see for example Lemma \ref{indcomb}.
Terao's longstanding conjecture states that for a fixed field, freeness of the module of derivations of an arrangement of hyperplanes is a combinatorial property of its intersection lattice (see \cite{p-hT-80} or \cite{p-hT-83}).

Although inductive freeness is a powerful tool to verify the freeness of many interesting arrangements, there are probably even more arrangements, which are free but not inductively free (see \cite[Example 4.59]{OT} for one of the oldest examples).
But in fact, the Addition-Deletion-Theorem allows to prove the freeness of a much bigger class of arrangements introduced in \cite[Def.\ 3.6.4]{gZ-87}\footnote{The original definition by Ziegler is slightly different than \cite[Def.~4.60]{OT}, but it coincides with our definition in dimension three.}, the \emph{recursively free} arrangements (see also Def.\ \ref{def:recfree}).
An attempt to settle Terao's conjecture is to answer the question whether every free arrangement is recursively free (see \cite[3.6]{gZ-87}, \cite[4.3]{OT}, \cite[5]{p-S-99}).
However, this is not the case: In this note we present free but not recursively free arrangements in characteristic zero, and complete the picture:
\begin{center}
inductively free \quad$\subsetneq$\quad recursively free \quad$\subsetneq$\quad free.
\end{center}
We first found our counterexample $\Ac$ with $27$ hyperplanes in $\CC^3$ by using the enumeration techniques introduced in \cite{p-CG-13} and the realization algorithm from \cite{p-C10b}.

The intersection lattice of $\Ac$ has a further realization $\Bc$ over $\FF_{11}$. It turns out that $\Bc$ is free and recursively free although it has the same intersection lattice as $\Ac$. Hence recursive freeness is not a purely combinatorial property of the intersection lattice, and should thus be perceived in a different way than the notion of inductive freeness.
This observation is implicit in the examples in \cite[4]{p-gZ-90} for arrangements over finite fields: Ziegler defines a matroid which has a realization which is free or not free depending on the chosen field, but is of course never inductively free. However, all his arrangements are recursively free when they are free.

Last but not least, the intersection lattice of $\Ac$ has realizations over $\FF_q$ for $q$ large enough (for example $q=251$) which are free but not recursively free. Thus the inclusion ``recursively free $\subsetneq$ free'' also holds for (at least certain) positive characteristics.

\section{Preliminaries}

We shortly review the required notions, compare with \cite{OT}.

\begin{defin}
Let $\Ac$ be an arrangement of hyperplanes, i.e.\ a finite set of hyperplanes in a fixed vector space $V$ over a field $K$.
Let $S=S(V^*)$ the symmetric algebra of the dual space $V^*$ of $V$.
We choose a basis $x_1,\ldots,x_r$ for $V^*$ and identify $S$ with
$K[x_1,\ldots,x_r]$ via the natural isomorphism $S\cong K[x_1,\ldots,x_r]$.
We write $\Der(S)$ for the set of derivations of $S$ over $K$.
It is a free $S$-module with basis $D_1,\ldots,D_r$ where $D_i$ is the usual derivation
$\partial/\partial x_i$.

A nonzero element $\theta\in\Der(S)$ is \emph{homogeneous of polynomial degree} $p$
if $\theta=\sum_{k=1}^r f_k D_k$ and $f_k\in S_p$ for $a\le k\le r$.
In this case we write $\pdeg \theta = p$.

Let $\Ac$ be an arrangement in $V$ with defining polynomial
\[ Q(\Ac) = \prod_{H\in\Ac} \alpha_H \]
where $H=\ker \alpha_H$, $\alpha_H\in V^*$. Define the \emph{module of $\Ac$-derivations} by
\[ D(\Ac) = \{\theta\in\Der(S)\mid \theta(Q(\Ac))\in Q(\Ac)S\}. \]
An arrangement $\Ac$ is called a \emph{free arrangement} if $D(\Ac)$ is a free
module over $S$.
\end{defin}

If $\Ac$ is free and $\{\theta_1,\ldots,\theta_r\}$ is a homogeneous basis for
$D(\Ac)$, then $\pdeg\theta_1,\ldots,\pdeg\theta_r$ are called the \emph{exponents}
of $\Ac$ and we write
\[ \exp \Ac = \{\{\pdeg\theta_1,\ldots,\pdeg\theta_r\}\}, \]
where the notation $\{\{*\}\}$ is used to emphasize the fact that it is a multiset.
Remark that the exponents depend only on $\Ac$.

\begin{defin}[{\cite[1.12-1.14]{OT}}]
Let $(\Ac,V)$ be an arrangement. We denote $L(\Ac)$ the set of all nonempty
intersections of elements of $\Ac$ including the empty intersection $V$.

If $\Bc\subseteq\Ac$ is a subset, then $(\Bc,V)$ is called
a {\it subarrangement}. For $X\in L(\Ac)$ define a subarrangement $\Ac_X$ of $\Ac$ by
\[ \Ac_X = \{H\in\Ac\mid X\subseteq H\}. \]
Define an arrangement $(\Ac^X,X)$ in $X$ by
\[ \Ac^X=\{X\cap H\mid H\in\Ac\backslash\Ac_X \mbox{ and } X\cap H\ne \emptyset\}.\]
We call $\Ac^X$ the {\it restriction} of $\Ac$ to $X$.

Let $H_0\in\Ac$. Let $\Ac'=\Ac\backslash\{H_0\}$ and let $\Ac''=\Ac^{H_0}$.
We call $(\Ac,\Ac',\Ac'')$ a {\it triple} of arrangements and $H_0$ the
{\it distinguished} hyperplane.
\end{defin}

We will use the following important theorem:

\begin{theor}[Addition-Deletion, {\cite[Thm.~4.51]{OT}}]\label{adddel}
Suppose $\Ac\ne\emptyset$. Let $(\Ac,\Ac',\Ac'')$ be a triple.
Any two of the following statements imply the third:
\begin{eqnarray*}
\Ac \mbox{ is free with } \exp \Ac &=& \{\{b_1,\ldots,b_{r-1},b_r\}\}, \\
\Ac' \mbox{ is free with } \exp \Ac' &=& \{\{b_1,\ldots,b_{r-1},b_r-1\}\}, \\
\Ac'' \mbox{ is free with } \exp \Ac'' &=& \{\{b_1,\ldots,b_{r-1}\}\}. \\
\end{eqnarray*}
\end{theor}

Inspired by this theorem, one defines:

\begin{defin}[{\cite[Def.~4.53]{OT}}]\label{def:indfree}
The class $\IF$ of \emph{inductively free} arrangements is the smallest
class of arrangements which satisfies
\begin{enumerate}
\item The empty arrangement $\Phi_\ell$ of rank $\ell$ is in $\IF$ for $\ell\ge 0$,
\item\label{part2} if there exists $H\in \Ac$ such that $\Ac''\in\IF$, $\Ac'\in\IF$, and
$\exp \Ac''\subset\exp \Ac'$, then $\Ac\in\IF$.
\end{enumerate}
\end{defin}

\begin{lemma}\label{indcomb}
The property of an arrangement $\Ac$ of being inductively free is a 
combinatorial property of its intersection lattice $L(\Ac)$.
\end{lemma}
\begin{proof} All information needed for part (\ref{part2}) in Def.\ \ref{def:indfree} is
included in the intersection lattice $L(\Ac)$: For $H \in \Ac$,
the intersection lattices $L(\Ac\backslash\{H\})$ and $L(\Ac^H)$ can be obtained as sublattices of $L(\Ac)$.
\end{proof}

A class of arrangements which is bigger than the class of inductively free ones is:

\begin{defin}[{\cite[Def.~4.60]{OT}}]\label{def:recfree}
The class $\RF$ of \emph{recursively free} arrangements is the smallest
class of arrangements which satisfies
\begin{enumerate}
\item The empty arrangement $\Phi_\ell$ of rank $\ell$ is in $\RF$ for $\ell\ge 0$,
\item if there exists $H\in \Ac$ such that $\Ac''\in\RF$, $\Ac'\in\RF$, and
$\exp \Ac''\subset\exp \Ac'$, then $\Ac\in\RF$,
\item if there exists $H\in \Ac$ such that $\Ac''\in\RF$, $\Ac\in\RF$, and
$\exp \Ac''\subset\exp \Ac$, then $\Ac'\in\RF$.
\end{enumerate}
\end{defin}

For $\alpha\in V^*$, we will write $\alpha^\perp$ for the kernel of $\alpha$.

\section{The counterexamples}

We will use the following simple lemma (see also \cite[Cor.\ 2.18]{p-HR12}):

\begin{lemma}\label{strnotindfree}
Let $\Ac$ be a free arrangement in $K^3$ with exponents $\{\{1,e,f\}\}$.
If $|\Ac^H| \notin\{e+1,f+1\}$ for all $H\in\Ac$, then $\Ac$ is not inductively free.
\end{lemma}
\begin{proof}
Assume that $|\Ac^H| \notin\{e+1,f+1\}$ for all $H\in\Ac$. If $\Ac$ was inductively free, then there would exist a hyperplane $H\in\Ac$ such that $(\Ac,\Ac\backslash\{H\},\Ac^H)$ is a triple of arrangements with $\exp \Ac^H\subset\exp \Ac$. But then the exponents of $\Ac^H$ would be either $(1,e)$ or $(1,f)$, and hence $|\Ac^H| \in\{e+1,f+1\}$ which is a contradiction.
\end{proof}

\begin{defin}
Let $\zeta$ be a fifth root of unity in $\CC$ and $\omega=-\zeta^2-\zeta^3$ be the golden ratio.
The Coxeter group $W$ of type $H_3$ may be generated as a reflection group by the reflections (see for example \cite{b-Hum})
\[ g_1:=\begin{pmatrix} 1&0&0\\ 0&1&1\\ 0&0&-1 \end{pmatrix},\quad
g_2:=\begin{pmatrix} -1&0&0\\ \omega &1&0 \\ 0&0&1 \end{pmatrix},\quad
g_3:=\begin{pmatrix} 1&\omega &0 \\ 0&-1&0 \\ 0&1&1 \end{pmatrix} \]
acting on $V^*\cong K^3$, $K=\QQ(\zeta)$.
Now let $\Phi^+\in V^*$ be the corresponding set of positive roots of $W$ (the orbit of the standard basis under $W$) and
\[ R := \Phi^+ \:\:\dot\cup\:\: (1,-\zeta^2,0)\cdot W. \]
Then $\Ac:=\{v^\perp\mid v\in R\}$ is an arrangement with $27$ hyperplanes:
The set $\Phi^+$ has $15$ elements and the other orbit has $12$ hyperplanes.
\end{defin}

\begin{remar}
The image of $W$ under the canonical map $\GL_3(\CC)\rightarrow \PGL_3(\CC)$ is isomorphic to the alternating group $A_5$.
Indeed, $\Ac$ may also be obtained as the union of two orbits under the action of a three dimensional representation of $A_5$.
\end{remar}

\begin{remar}
The minimal field of definition over $\QQ$ for a realization of the intersection lattice of $\Ac$ is $\QQ(\zeta)$: There is no arrangement over $\QQ(\omega)$ with the same intersection lattice than $\Ac$.
\end{remar}

\begin{theor}\label{H3Anotrecfree}
The arrangement $\Ac$ is free but not recursively free.
\end{theor}
\begin{proof}
Factorizing the characteristic polynomial gives
\[ \chi_\Ac(t)=(t-1)(t-11)(t-15). \]
One can now prove the freeness of $\Ac$ using \cite[Thm.\ 1.39 (ii)]{p-Y-12} or \cite[Cor.\ 3.3]{p-Y-05}\footnote{Alternatively, a direct computation with {\sc Magma}, \cite{MAGMA} verified once again with {\sc Singular}, \cite{SIN} tells us that $\Ac$ is free.}:
Choose a hyperplane $H\in\Ac$ and compute the exponents $\{\{d_1,d_2\}\}$ of the multiarrangement $(\Ac^H,\mm^H)$. It turns out that $d_1d_2=165$, thus $\Ac$ is free.

Inspection of the intersection lattice gives the following multisets of invariants:
\begin{eqnarray*}
\{\{|\Ac^H| \mid H\in\Ac \}\} &=& \{\{10^{15},11^{12}\}\}, \\
\{\{|\Ac_p| \mid p\in L(\Ac),\:\:\dim p=1 \}\} &=& \{\{2^{15}, 3^{70}, 7^6\}\}.
\end{eqnarray*}
Notice first that each hyperplane in $\Ac$ contains either $10$ or $11$ intersection points, thus $\Ac$ is not inductively free by Lemma \ref{strnotindfree}.

Now assume that we include a new hyperplane $H$ to $\Ac$, so $\tilde\Ac:=\Ac\dot\cup\{H\}$, and assume that $(\tilde\Ac,\Ac,\tilde\Ac^H)$ is a triple of arrangements with $\exp \tilde\Ac^H\subset\exp \tilde\Ac$, i.e.\ the exponents of $\tilde\Ac$ are either $\{\{1,11,16\}\}$ or $\{\{1,12,15\}\}$ depending on $H$.

If $H$ contains no intersection point $p\in L(\Ac)$, $\dim p=1$, then $|{\tilde\Ac}^H|$ will be equal to $|\Ac|=27$, contradicting $\exp \tilde\Ac^H\subset\exp \tilde\Ac$. If it contains exactly one such point, then $|{\tilde\Ac}^H|\ge 21$ by the above computation of the numbers $|\Ac_p|$, again a contradiction.

There are only 1186 cases left in which $H$ contains at least two intersection points of $L(\Ac)$. None of these $1186$ arrangements is free: This may either be verified by factorizing their characteristic polynomial or by a direct computation. Thus the assumption $\exp \tilde\Ac^H\subset\exp \tilde\Ac$ for the triple $(\tilde\Ac,\Ac,\tilde\Ac^H)$ is false. It turns out that there is now way up or down via the Addition-Deletion-Theorem starting at $\Ac$ and staying in the same dimension, i.e.\ $\Ac$ is not recursively free.
\end{proof}

\begin{remar}
The arrangement $\Ac$ has the `same' intersection lattice as the arrangement
\begin{eqnarray*}
\Bc &=& \{(0,0,1)^\perp,(0,1,1)^\perp,(0,1,2)^\perp,(0,1,3)^\perp,(0,1,4)^\perp,(0,1,5)^\perp, \\
    &&    (0,1,6)^\perp,(1,0,0)^\perp,(1,0,1)^\perp,(1,2,2)^\perp,(1,3,1)^\perp,(1,3,10)^\perp, \\
    &&    (1,4,3)^\perp,(1,4,4)^\perp,(1,5,7)^\perp,(1,6,4)^\perp,(1,6,6)^\perp,(1,8,5)^\perp, \\
    &&    (1,8,8)^\perp,(1,9,0)^\perp,(1,9,1)^\perp,(1,9,4)^\perp,(1,9,5)^\perp,(1,9,8)^\perp, \\
    &&    (1,9,9)^\perp,(1,10,0)^\perp,(1,10,5)^\perp \}.
\end{eqnarray*}
in $\FF_{11}^3$. As for $\Ac$, the arrangement $\Bc$ is a union of two orbits under the action of a subgroup of $\PGL_3(\FF_{11})$ isomorphic to $A_5$.

Notice that contrary to $\Ac$, the arrangement $\Bc$ is recursively free: Let
\begin{eqnarray*}
(H_1,\ldots,H_7) &:=& ((0,1,0)^\perp,(1,1,0)^\perp,(1,2,0)^\perp,(1,3,0)^\perp, \\
 && (1,4,0)^\perp,(1,5,0)^\perp,(1,7,0)^\perp).
\end{eqnarray*}
Then the arrangements $\Bc_i:=\Bc\cup \{H_1,\ldots,H_i\}$, $i=1,\ldots,7$ are free, and $\Bc_7$ is inductively free.
\end{remar}

\begin{remar}
The reflection arrangement $\Ac$ of the complex reflection group $G_{27}$ has $45$ hyperplanes and is also free but not recursively free.
It is free with exponents $\{\{1,19,25\}\}$ for instance by \cite[Thm.\ 6.60]{OT}. The proof that it is not recursively free is similar to the proof of Thm.\ \ref{H3Anotrecfree}: It is not inductively free by inspection of the intersection lattice:
\begin{eqnarray*}
\{\{|\Ac^H| \mid H\in\Ac \}\} &=& \{\{16^{45}\}\}, \\
\{\{|\Ac_p| \mid p\in L(\Ac),\:\:\dim p=1 \}\} &=& \{\{3^{120}, 4^{45}, 5^{36}\}\}.
\end{eqnarray*}
As in Thm.\ \ref{H3Anotrecfree}, one can then consider all possibilities of including a new hyperplane, although the number of possible new hyperplanes is much bigger here. 
\end{remar}

\begin{remar}
Joining four orbits of elements under a certain subgroup of $G_{25}$ (or $G_{26}$), one obtains a free but not recursively free arrangement with $39$ hyperplanes.
\end{remar}


\def\cprime{$'$}
\providecommand{\bysame}{\leavevmode\hbox to3em{\hrulefill}\thinspace}
\providecommand{\MR}{\relax\ifhmode\unskip\space\fi MR }
\providecommand{\MRhref}[2]{%
  \href{http://www.ams.org/mathscinet-getitem?mr=#1}{#2}
}
\providecommand{\href}[2]{#2}

\end{document}